\documentclass[letterpaper, 10 pt, conference]{ieeeconf}  

\IEEEoverridecommandlockouts                            
\let\labelindent\relax

\usepackage[utf8]{inputenc}
\usepackage{mathtools, amssymb}

\usepackage[ruled,vlined]{algorithm2e}
\usepackage{tikz}
\usepackage{xcolor}
\usepackage{siunitx}
\usepackage[compatibility=false]{caption}
\usepackage{subcaption}

\usepackage[shortlabels]{enumitem}

\usetikzlibrary{positioning}
\usetikzlibrary{arrows, arrows.meta}
\usepackage[backend=biber]{biblatex}
\DeclareLanguageMapping{english}{english-apa}
\usepackage{hyperref}
\usepackage{graphicx}

\allowdisplaybreaks

\author{Behnam Mafakheri$^{1}$, Jonathan H. Manton$^{1}$ and Iman Shames$^{2}$% <-this % stops a space
\thanks{*This work was supported by the Australian Research Council under the Discovery Projects funding scheme (DP210102454).}% <-this % stops a space
\thanks{$^{1}$B. Mafakheri and J. H. Manton are with Department of Electrical and Electronic Engineering,
        University of Melbourne, VIC 3010, Australia
        {\tt\small \{mafakherib, jmanton\}@unimelb.edu.au}}%
\thanks{$^{2}$I. Shames is  with the The CIICADA Lab, School of Engineering, The Australian National University, Canberra, ACT 2601, Australia
        {\tt\small iman.shames@anu.edu.au}}%
}

\date{\today}
\newtheorem{theorem}{Theorem}
\newtheorem{definition}[theorem]{Definition}
\newtheorem{remark}[theorem]{Remark}
\newtheorem{lemma}[theorem]{Lemma}

\DeclareMathOperator*{\argmin}{arg\,min}
\newcommand{\prox}{\mathrm{Prox}}

\newcommand{\RR}{\mathbb{R}}

\newcommand{\NN}{\mathbb{N}}

\newcommand{\cE}{{\mathcal{E}}}

\newcommand{\cK}{{\mathcal{K}}}
\newcommand{\cL}{{\mathcal{L}}}

\newcommand{\cN}{{\mathcal{N}}}
\newcommand{\cO}{{\mathcal{O}}}

\newcommand{\cV}{{\mathcal{V}}}

\newcommand{\va}{{\boldsymbol{a}}}
\newcommand{\vb}{{\boldsymbol{b}}}

\newcommand{\vs}{{\boldsymbol{s}}}
\newcommand{\vt}{{\boldsymbol{t}}}
\newcommand{\vu}{{\boldsymbol{u}}}
\newcommand{\vv}{{\boldsymbol{v}}}

\newcommand{\vx}{{\boldsymbol{x}}}
\newcommand{\vy}{{\boldsymbol{y}}}

\newcommand{\vlambda}{{\boldsymbol{\lambda}}}
\newcommand{\veps}{{\boldsymbol{\epsilon}}}
\newcommand{\veta}{{\boldsymbol{\eta}}}

\title{\LARGE \bf On Distributed Nonconvex Optimisation Via Modified ADMM}
\begin{document}

\maketitle
\thispagestyle{empty}
\pagestyle{empty}

\begin{abstract}
    This paper addresses the problem of nonconvex nonsmooth decentralised optimisation in multi-agent networks with undirected connected communication graphs. Our contribution lies in introducing an algorithmic framework designed for the distributed minimisation of the sum of a smooth (possibly nonconvex and non-separable) function and a convex (possibly nonsmooth and non-separable) regulariser. The proposed algorithm can be seen as a modified version of the ADMM algorithm where, at each step, an ``inner loop'' needs to be iterated for a number of iterations. The role of the inner loop is to aggregate and disseminate information across the network. We observe that a naive decentralised approach (one iteration of the inner loop) may not converge. We establish the asymptotic convergence of the proposed algorithm to the set of stationary points of the nonconvex problem where the number of iterations of the inner loop increases logarithmically with the step count of the ADMM algorithm. We present numerical results demonstrating the proposed method's correctness and performance.
\end{abstract}

\section{Introduction}
    While distributed solutions to large-scale convex optimisation problems have been studied extensively in the last two decades, the same cannot be said about the nonconvex variant~\cite{tatarenko2017non, scutari2018parallel}. We propose a distributed algorithm for solving problems of the following archetypal form:
\begin{align}\label{my_P}
    % \tag{Consensus Optimisation}
    \begin{array}{ll}
        \underset{\vx \in \RR^p}{\mbox{minimize}} &\displaystyle{f(\vx):=\sum_{i=1}^n f_i(\vx) + g(\vx)},   
    \end{array}
\end{align}
where each $f_i$ represents smooth (possibly nonconvex, non-separable) functions, while $g$ is a convex (possibly nonsmooth) function. The objective is to have $n$ agents cooperate in solving the optimisation problem over an underlying fixed communication graph. It is assumed that each agent $i$ privately possesses knowledge of the local objective function $f_i(\vx)$ and the function $g(\vx)$, where $g$ is
used to impose desired
structures on the solution (e.g., $\ell_1$ norm for sparsity) and/or used to enforce certain constraints. In machine learning applications, this controls the complexity of the model to avoid over-fitting in the training process \cite{chang2016asynchronous}.

The distributed algorithms aimed at solving \eqref{my_P} can be classified into either (sub)gradient-based methods or operator splitting-based methods, applied to either the primal or dual problems. In this work, we explore the Alternating Direction Method of Multipliers (ADMM), a powerful operator splitting-based method widely recognised for its effectiveness in numerically solving optimisation problems \cite{boyd2011distributed, deng2016global, hong2016convergence}. However, existing distributed ADMM algorithms face limitations in their applicability, as they either assume a communication network structured as a star graph with a central node or overlook the role of the regularizing function $g$. In scenarios where a central node is present in the graph, previous works \cite{hong2016convergence, chang2016asynchronous, hong2017distributed} have shown that ADMM converges in asynchronous settings to a stationary point of the problem with a sublinear rate. Modified ADMM versions tailored for arbitrary connected communication graphs have been proposed by authors in \cite{pmlr-v70-hong17a, yi2022sublinear}. Specifically, the Prox-PDA algorithm in \cite{pmlr-v70-hong17a} achieves sublinear convergence to a stationary point of the problem \eqref{my_P} with $g\equiv 0$.

In a similar problem context, \cite{yi2022sublinear} proves the convergence of a modified ADMM version to the global minimum under the condition that the global cost function satisfies the Polyak-{\L}ojasiewicz condition. However, it is crucial to emphasize that these existing schemes and convergence analyses are not directly transferable to address our specific problem \eqref{my_P} when $g \neq 0$, as they do not consider the presence of the regulariser function, potentially rendering the objective function nonsmooth. To overcome this limitation, our work introduces an ADMM-based algorithm explicitly designed to decentralise the optimisation algorithm. This enables the algorithm's applicability to arbitrarily connected communication graphs while accommodating convex regulariser functions $g$. Our contributions can be summarised as follows:
\begin{itemize}
    \item The proposed algorithm is applicable to any arbitrary graph that is connected and extends to problems with convex regulariser functions $g$.
    \item We prove the algorithm's convergence to the set of stationary points of the problem, provided that the $f_i$'s are smooth and $g$ is a closed convex function.
    \item Instead of assuming a central node in the graph, we propose that nodes emulate the existence of such a central node by approximating the information that would have been provided through interactions with their immediate neighbours.
\end{itemize}

The remainder of this paper is organised as follows: Section \ref{preliminaries} introduces the necessary notation and preliminary definitions. The main results and the proposed algorithm are discussed in Section \ref{main_results}, while Section \ref{proof} is dedicated to the accompanying proofs. Section \ref{numerical_results} provides the numerical results, validating the effectiveness and efficiency of the proposed approach. In Section \ref{conclusion}, we summarise our contributions and outline future research directions.

\section{Preliminaries} \label{preliminaries}
    \subsubsection*{Notations}
    Throughout this paper, we denote the set of real numbers and extended real numbers as $\RR$ and $\overline{\RR}:=\RR \cup \{+\infty\}$. Let $[n]$ denote the set of $n$ positive integers $\{1, 2, \dots, n \}$. For vectors $\va, \vb\in \RR^n$ the Euclidean inner product and its corresponding norms are denoted by $\langle \va , \vb \rangle$ and $\| \va\|$ respectively. For a matrix $A\in \RR^{m\times n}$, $\| A\|$ is the spectral norm of the matrix. A vector of all ones with appropriate length is shown by $\boldsymbol{1}$. We say a matrix $W$ is row (column) stochastic if $W \boldsymbol{1} = 1$ ($\boldsymbol{1}^T W = \boldsymbol{1}^T$). A doubly stochastic matrix is defined as both column and row stochastic.
    % We show the $\limsup$ and $\liminf$ of a sequence by $\overline{\lim}$ and ${\underline{\lim}}$, respectively.

\subsubsection*{Graphs}
    A graph $G$ is determined by its set of vertices and edges. We write $G=(\cV, \cE)$, where $\cV = \{1,\dots, n\}$ is the set of nodes and $\cE \subseteq \cV\times \cV$ is the set of edges. The presence of edge $(i,j)\in \cE$ indicates that nodes $i$ and $j$ are adjacent and node $i$ can receive information from node $j$. In this study, we assume that the graph $G$ is undirected and connected, i.e., if $(i,j)\in \cE$, then $(j,i)\in \cE$ and there is a path between every pair of nodes. The set of neighbours of node $i$ is denoted by $\cN_i: = \{j \mid (i,j)\in \cE \}$. We also write $\overline{\cN}_i = \cN_i \cup \{i\} $. A matrix $W$ is a \emph{weight matrix} associated with the graph if for some $\eta >0$ we have $w_{ij} > \eta$ whenever $j\in \overline{\cN}_i$ and $w_{ij}=0$ otherwise. The following lemma plays an important role in our later analysis.
    \begin{lemma}\label{fix_graph_consensus_rate}
        (\cite[Proposition 2]{nedic2009distributed}, \cite[Lemma 5.2.1]{tsitsiklis1984problems}) For a doubly stochastic weight matrix $W$ associated with a connected graph $G$ with $n$ nodes, there exist some $c >0$ and $\rho \in (0,1)$ such that for all $m\geq 1$,
            $\|W^m - n^{-1}\boldsymbol{1} \boldsymbol{1}^T \| \leq c \rho ^m.$
    \end{lemma}

\subsubsection*{Subdifferentials and Proximal Maps}
    For indicating a set-valued mapping we use $A:\RR^p \to \RR^q$ that maps a point $\vx \in \RR^p$ to a set $A(\vx)\subset \RR^q$. The \emph{graph} of $A$ is defined as $gra \  A := \{(\vx, \vu) \in \RR^p \times \RR^q\ | \  \vu \in A(\vx) \}$. The \emph{domain} and \emph{epigraph} of an extended real-valued function $f: \RR^p \to \overline{\RR}$ are defined as the sets $\text{dom} \ f := \{\vx\in \RR^p \ | \ f(\vx) < +\infty \}$ and $\text{epi} \ f := \{(\vx, t) \in \RR^p \times \RR \ | \ f(\vx) \leq t \}$, respectively. A function $f$ is called proper if $f(\vx)<\infty$ for at least one $\vx\in \RR^p$ and $f(\vx)>-\infty$ for all $\vx\in \RR^p$. It is said to be \emph{closed} or equivalently \emph{lower semicontinuous (lsc)} if its epigraph is a closed set in $\RR^{p+1}$, see \cite[Thm. 1.6]{rockafellar1998VariationalAnalysis}. 
    Now, we recall some definitions in relation to subdifferential calculus.
    \begin{definition}
        (Subdifferentials, \cite[Ch. 8]{rockafellar1998VariationalAnalysis}) Let $f:\RR^p \to \overline{\RR}$ be a proper lsc function.
        \begin{enumerate}[(i)]
            \item For $\vx\in \text{dom}\ f$, the \emph{Frechet-subdifferential} (also known as the \emph{regular-subdifferential}) of $f$ at $\vx$, denoted by $\widehat{\partial}f(x)$, is defined by
            \begin{align*}
                \widehat{\partial}f(x)&:=\{\vs\in \RR^p : \\
                &\liminf_{\vx' \to \vx, \vx'\neq \vx}{\frac{f(\vx')- f(\vx) - \langle \vs, \vx'- \vx \rangle}{\|\vx'-\vx \|}}\geq 0 \}.
            \end{align*}
            If $\vx \not\in \text{dom} f$, define $\widehat{\partial} f(\vx) = \emptyset$.
            \item For $\vx\in \RR^p$, the \emph{limiting-subdifferential}, or \emph{subdifferential}, of $f$ at $\vx$, denoted by $\partial f(x)$ is defined as
            \begin{align*}
                \partial f(\vx)&:= \{\vs \in \RR^p : \exists \vx^r\to \vx, \\
                &\  f(\vx^r) \to f(\vx),\exists \vs^r \in \widehat{\partial}f(\vx^r)\to \vs \ \text{as}\  r\to \infty \}.
            \end{align*}
        \end{enumerate}
    \end{definition}
    \begin{remark} \label{subdiff_properties}
        % \begin{enumerate}[(i)]
        % \item For any $\vx\in\RR^n$, the definition above implies $\widehat{\partial}f(\vx) \subset \partial f(\vx)$ where the first set is closed and convex while the second is closed.
        % \item 
        For sequences $\{\vx\}_{r\in \NN}$ and $\{\vs\}_{r\in \NN}$ satisfying $\vs^r \in \partial f(\vx^r)$ and $(\vx^r, f(\vx^r), \vs^r) \to (\vx, f(\vx), \vs)$ as $r \to \infty$, then $\vs\in \partial f(\vx)$. \label{closeness_subdiff}
        % \item (Fermat's rule) If $\vx\in \RR^n$ is a local minimiser of $f$, then $0\in \partial f(\vx)$. Points that $0$ is in their subdifferential are called \emph{critical points}.
        % \item If $f$ is convex, the set of critical points and minimisers coincide and
        % \begin{align*}
        %     &\widehat{\partial}f(\vx) = \partial f(\vx)= \\
        %     & \ \{\vs\in \RR^n :
        %     f(\vx') \geq f(\vx) + \langle \vs, \vx'-\vx \rangle, \  \forall \vx' \in \RR^n \}.
        % \end{align*}
        % \end{enumerate}
    \end{remark}

    % We say a function $f$ is \emph{prox-bounded} if there is a $\gamma>0$ such that $f+\frac{1}{2\gamma}\|\cdot \|^2$ is lower bounded. The supremum of all such $\gamma$ is denoted by $\gamma_f$. 
    The \emph{proximal mapping} of $f$ with parameter $\gamma$ is a set valued mapping defined as $\prox_{\gamma f}(\vx) : = \argmin_{\vu\in \RR^p} \{f(\vu) + \frac{1}{2 \gamma} \|\vu - \vx \|^2\}$. 
    % We have the following lemma on the well-definedness of the proximal operator.
    % \begin{lemma}
    %     (\cite[Theorem 1.25]{rockafellar1998VariationalAnalysis}) if function $f$ is proper, lower semi-continuous, and prox-bounded, then $\prox_{\gamma f}$ is a nonempty and compact subset of $\RR^n$ for $\gamma \in (0, \gamma_f)$.
    % \end{lemma}
    The result below can be seen from the necessary optimality condition of the problem defining $\prox_{\gamma f}$.
    \begin{lemma}\label{prox_map_optimality_cond}
        If $\vu\in \prox_{\gamma f}(\vx)$ then $\vx - \vu \in \gamma \partial f(\vu)$.
    \end{lemma}
    \begin{lemma}\label{prox_properties_convex}
        (\cite{parikh2014proximal}) If the function $f$ is proper, lower semicontinuous and convex, then the following hold.
        \begin{enumerate}[(i)]
        \item $\prox_{\gamma f}$ is single-valued for every $\gamma >0$.
        \item $\prox_f$ is firmly nonexpansive \cite[Sec 2.3]{ryu2022large} and therefore Lipschitz continuous with $L=1$.
        % i.e.
        % \begin{align*}
        %     &\| \prox_f(\vx_2) - \prox_f(\vx_1)\|^2 \leq \\
        %     &\qquad \langle \prox_f(\vx_2) - \prox_f(\vx_1),\ \vx_2 - \vx_1 \rangle.
        % \end{align*} \label{prox_firm_nonexp}
        \end{enumerate}
    \end{lemma}
    
    We write $f\in C^{1,1}_L$ to indicate the class of functions $f:\RR^p \to {\RR}$ that are differentiable and have Lipschitz continuous gradients with parameter $L$. For simplicity, we say that such an $f$ is \emph{L-smooth} or \emph{smooth}. We have the following important lemma on smooth functions.
    
    \begin{lemma} \label{descent_lemma_1} 
        (Descent Lemma \cite[Proposition~A.24]{bertsekas1997nonlinear}) Let the function $f: \RR^p \to {\RR}$ be an $L$-smooth function. Then for every $\vx$, $\vy \in \RR^p$ the following holds
        \begin{align*}
            f(\vy) \leq f(\vx) + \langle \nabla f(\vx), \vy - \vx\rangle + {L}/{2} \|\vy - \vx \|^2.
        \end{align*}
    \end{lemma}

\section{Algorithm And The Main Results} \label{main_results}
This section presents a decentralised ADMM-based algorithm for solving \eqref{my_P} under the following assumptions:
    \begin{enumerate}[label=(A\arabic*)]

        \item \label{smotthnes_f_i} Each $f_i:\RR^p\to \RR$ is in $C_L^{1,1}$ for all $i\in [n]$.
        
        \item \label{proper_f_i} The functions $f_i$ and $g$ are proper and the function $f:\RR^p \to \overline{\RR}$ is lower bounded. 
        
        \item \label{g_X_assump} The function $g(\vx)$ is convex and possibly nonsmooth.
        
        \item \label{graph_connec_assump} The underlying communication graph, $G=(\mathcal{V}, \mathcal{E})$, is undirected, connected, and fixed. 
    
    \end{enumerate}
The (centralised) ADMM algorithm with penalty parameter $\beta$ applied to \eqref{my_P} is
\begin{subequations}\label{eq:ADMM_updates}
    \begin{align}
        \widetilde{\vx}^{r+1} \!=& n^{-1}\sum_{j=1}^n (\vx_j^r + {\vlambda_j^r/\beta}), \label{eq: central step}\\
        \vx_0^{r+1} \in& \prox_{\frac{1}{n\beta} g}(\widetilde{\vx}^{r+1}) ,  \label{x_0_update}\\
        \vx_i^{r+1} \!=& \argmin_{\vx_i\in \RR^{p}} f_i(\vx_i) + \langle \vx_i, \vlambda_i^r \rangle + \frac{\beta}{2}\|\vx_i - \vx_{0}^{r+1} \|^2, \label{x_i_update}\\
        \vlambda_i^{r+1} \!=& \vlambda_i^{r} + \beta(\vx_i^{r+1} - \vx_{0}^{r+1}), \label{dual_update}
    \end{align}
\end{subequations}
where $i\in[n]$. While the last two iterations can be done in parallel for $n$ agents, the algorithm requires a central node to perform the first update \eqref{x_0_update}. To create a decentralised algorithm, each node should be able to estimate $\widetilde{\vx}^{r+1}$ locally. The naive algorithm, where each agent takes average over its neighbours instead of the step \eqref{eq: central step}, will not necessarily converge. Therefore, we add an inner loop for estimating $\widetilde{x}^{r+1}$ where an $\epsilon$-consensus algorithm is used as described in Algorithm \ref{alg:concensus}. 

\SetKwComment{Comment}{/* }{ */}

\begin{algorithm}
    \caption{$\epsilon$-consensus algorithm}\label{alg:concensus}
    \KwData{$\epsilon, G, n, W:=[w_{ij}], \{\vv_i\}_{i=1}^n$}
    \KwResult{$\{\vv_{0,i} \}_{i=1}^n$}
    $\vv_i^0 \gets \vv_i$\ and $\bar{\vv} :={n}^{-1}\sum_{i=1}^n \vv_i$;
    
    \For {$l=0, 1, 2, \dots$}{
    
    \For {$i=1,\dots , n$}
        {   
            in parallel 
            
            \While{$\|\vv_i^l - \bar{\vv} \| >\epsilon$}{
            $\vv_i^{l+1} \gets \sum_{i=1}^n w_{ij}\vv_i^l$}

        }
    }
    $\vv_{0,i} \gets \vv_i^{l}$
    \label{simple_consensus_Alg}
\end{algorithm}

    Assume graph $G$ satisfies \ref{graph_connec_assump} and $W$ is a doubly stochastic weight matrix associated with $G$. Then from Lemma \ref{fix_graph_consensus_rate} it is easy to see that for a given $\epsilon$ there is a $t_0(\epsilon)$ such that if $t > t_0(\epsilon)$ then the output of the Algorithm \ref{simple_consensus_Alg} satisfies $\|\vv_i^t - \bar{\vv} \| \leq \epsilon$ for all $i \in [n]$. This property of the algorithm can be used to have a local $\|\boldsymbol{\epsilon}_i^{r}\|$-approximation of $\vx_0^r$: 
\begin{align}\label{x_0_i_update_perturbed_with_g}
    \vx_{0,i}^{r+1} =\prox_{\frac{1}{n\beta}g}(\widetilde{\vx}^{r+1} + \veps_i^{r+1} ) \quad \forall i \in [n].
\end{align}
In the following, we present the Distributed ADMM Algorithm \ref{distributed_admm}. For simplicity of notation we denote the sequence $\{a_i\}_{i=1}^n$ by $\{a_i\}$. 
\begin{algorithm}
    \caption{Distributed ADMM}
    \KwData{$\beta, \{\veps_i^r\}_{r \in \NN}, G, n, W:=[w_{ij}], \{\vx_i^0, \vlambda_i^0\},\delta$}
    % \KwResult{$\{\vv_{0,i} \}_{i=1}^n$}
    
    \For {$r=0, 1, 2, \dots$}{
    \For {$i=1,\dots , n$}
    {   
        in parallel 
        
        %\While{\color{blue}\{\emph{equations \eqref{KKT1}, \eqref{KKT2}, and \eqref{KKT3} are satisfied up to a required accuracy\}}} {
        \While { $\max_{i\in[n]} \max \{ \|\nabla f_i(\vx_i^r) + \vlambda_i^r \|, \| \vs_{0,i}^r - n \widetilde{\vlambda}_i^r\|, \|\vx_i^r - \vx_{0,i}^r\|\} \geq \delta$ } {

        $(\widetilde{\vx}_i^{r+1} , \widetilde{\vlambda}_i^{r+1})\xleftarrow{\|\boldsymbol{\epsilon}_i^{r+1}\|-consensus}  \big((\{\vx_j^r\}_{j=1}^n , \{\vlambda_j^r\}_{j=1}^n), G, n, W\big)$,
        
        $\vy_{0,i}^{r+1} = \widetilde{\vx}_i^{r+1} + \widetilde{\vlambda}_i^{r+1}/\beta$

        $\vx_{0,i}^{r+1} \in \prox_{\frac{1}{n\beta}g}(\vy_{0,i}^{r+1})$
        
        $\vx_i^{r+1} \gets \argmin_{\vx_i\in \RR^{p}} f_i(\vx_i) + \langle \vx_i, \vlambda_i^r \rangle  + \frac{\beta}{2}\|\vx_i - \vx_{0,i}^{r+1} \|^2 $ 
        
        $\vlambda_i^{r+1} \gets \vlambda_i^{r} + \beta(\vx_i^{r+1} - \vx_{0,i}^{r+1})$
        }
    }
    }
    \label{distributed_admm}
\end{algorithm}
\begin{theorem}\label{main_theorem_sync}
    Assume that in iteration $r$ of Algorithm \ref{distributed_admm}, each agent stops the consensus algorithm after $t_r$ iterations such that $t_r \geq \frac{1+\zeta}{\log \rho^{-1}}\log{r}+ \frac{\log{c}}{\log{\rho^{-1}}}$ for some $\zeta >0$ and $\rho$ and $c$, as described in Lemma \ref{fix_graph_consensus_rate}. Then, under \ref{smotthnes_f_i}, \ref{proper_f_i}, \ref{g_X_assump}, and \ref{graph_connec_assump}, every limit point of the sequence $( \{\vx_i^r \}, \{\vlambda_i^r\}, \{\vx_{0, i}^r\})$ generated by the algorithm is a KKT point for the optimisation problem and satisfies the following
    \begin{subequations}
    \begin{align}
        &\nabla f_i(\vx_i^*) + \vlambda_i^* = 0 \quad \forall i\in[n], \label{KKT1}\\
        &\vs_{0,i}^* - \sum_{j=1}^n \vlambda_j^* = 0 \quad \forall i\in[n], \label{KKT2}\\
        &\vx_{0,i}^* =\vx_i^* = \vx_j^* \quad \forall i,j \in [n] \label{KKT3},
    \end{align}
    \end{subequations}
    where $\vs_{0,i}^* \in \partial g(\vx_{0,i}^*)$ ({see \cite{chang2016asynchronous}} and references therein for the KKT condition).
\end{theorem}
\begin{remark}

    If each agent stops the consensus algorithm after $t_r$ steps where $t_r$ satisfies the hypothesis of Theorem \ref{main_theorem_sync} then the sum $\sum_{r=0}^{\infty}\| \veps_i^{r+1}\|$ will be in form of $\sum_{r=0}^{\infty} \frac{1}{r^{1+\zeta}}$ since the value $\|\veps_i^r\|$ vanishes with $t_r$ exponentially. Thus, for all $i\in [n]$
    \[\sum_{r=0}^{\infty} \left(\|\vy_{0, i}^{r+1} - \widetilde{\vx}^{r+1} \|\right) = \sum_{r=0}^{\infty}\| \veps_i^{r+1}\|  < \infty.\]
    This, in turn, implies that the number of required communications is of order $\mathcal{O}(T\log{T})$ if the number of optimisation iterations is of order $\mathcal{O}(T)$.
\end{remark}

\begin{remark}
    The lower bound for $t_r$ depends on global information $\rho$ and $c$ that are a consequence of the underlying graph. In the case of fixed graphs, this information can be provided to all nodes before the start of the algorithm.
\end{remark}
\begin{remark}
    The stopping criterion of Algorithm~\ref{distributed_admm} corresponds to an approximate solution to the KKT system \eqref{KKT1}--\eqref{KKT3}. Note that $\nabla f_i(\vx_i^r) + \vlambda_i^r = 0$ for all $r$. Additionally, $\widetilde{\vlambda}_i^r$ is node $i$'s local estimate of $\bar{\vlambda}^r = \frac{1}{n}\sum_{i=1}^n \vlambda_i^r$ such that $\|\widetilde{\vlambda}_i^r - \bar{\vlambda}^r \| \leq \|\veps_i^r \|$, where $\|\veps_i^r\|$ is summable and therefore vanishing in $r$ for all $i\in [n]$. Finally, each node can readily check the last condition, $\|\vx_i^r - \vx_{0,i}^r\| <\delta$ as they have access to both $\vx_i^r$ and $\vx_{0,i}^r$. Implementation of the stopping criteria then would involve invoking a max-consensus algorithm that would terminate in at most $n$ steps. In a practical implementation, this condition needs only to be checked every $\bar{N}$ steps with ``a large enough'' $\bar{N}$ to reduce the communication overhead.
\end{remark}

\section{Proof of the main result} \label{proof}
The augmented-Lagrangian is defined as follows
\begin{align}
    \label{aug-lag-case2}
    &\cL \Big( \{\vx_i^r \}, \{\vlambda_i^r\}, \{\vx_{0, i}^r\} \Big) 
     = \sum_{i=1}^n \left(l_i^{\beta} (\vx_i^r, \vlambda_i^r, \vx_{0,i}^r)+\frac{1}{n}g(\vx_{0,i}^r)\right)
\end{align}
where $l_i^{\beta} (\vx_i^r, \vlambda_i^r, \vx_{0,i}^r) := f_i(\vx_i^r) + \langle \vlambda_i^r , \vx_i^r - \vx_{0,i}^r \rangle + \frac{\beta}{2} \|\vx_i^r - \vx_{0,i}^r \|^2$. 
We have the following important lemmas.
\begin{lemma}\label{aug-lag-bdd}
    Suppose \ref{smotthnes_f_i} and \ref{proper_f_i} hold. Then for every $\beta \geq L$, the augmented-Lagrangian defined in \eqref{aug-lag-case2} is lower bounded, i.e. there exists a scalar $\underline{\cL}$ such that $\cL^r \geq \underline{\cL}$.
\end{lemma}

\begin{proof}
    The first-order optimality condition of the right-hand side of \eqref{x_i_update} yields
    \begin{align} 
    \label{optim-cond-primal}
    \nabla f_i(\vx_i^{r+1}) + \vlambda_i^{r+1} = 0.
    \end{align}
    From \eqref{aug-lag-case2} we have
    \begin{align*}
        &\cL \Big( \{\vx_i^r \}, \{\vlambda_i^r\}, \{\vx_{i, 0}^r\} \Big) \\
        &\stackrel{(a)}{=} 
        \sum_{i=1}^n \Big( f_i(\vx_i^r)+n^{-1}g(\vx_{0,i}^r) + \langle \nabla f_i (\vx_i^{r}) , \vx_{0,i}^r - \vx_i^r \rangle   \\
        & \quad  + \frac{\beta}{2} \|\vx_i^r - \vx_{0,i}^r \|^2 \Big)\\
        &\stackrel{(b)}{\geq} \sum_{i=1}^n \left(f_i(\vx_{0,i}^{r}) + n^{-1}g(\vx_{0,i}^r) + \frac{\beta - L}{2}\|\vx_i^r - \vx_{0,i}^r  \|^2 \right)
        % \\
        % &{\color{red} = (wrong)}f(\vx_{0,i}^r) + \frac{\beta - L}{2}\sum_{i=1}^n \|\vx_i^r - \vx_{0,i}^r  \|^2,
    \end{align*}
    where (a) is a consequence of \eqref{optim-cond-primal} and in (b), we have used the descent Lemma \ref{descent_lemma_1} and gradient Lipschitz continuity of $f_i$'s. The claim follows from the Assumption \ref{proper_f_i} and the fact that $\beta \geq L$.
\end{proof}

\begin{lemma}\label{aug-lag-consecutive-change_with_g}
    Suppose the assumptions \ref{smotthnes_f_i}, \ref{proper_f_i}, and \ref{g_X_assump} hold. For the sequence generated by the Distributed ADMM algorithm, if $\beta > L$, then $\forall\vt_{0, i}^{r+1} \in \partial g(\vx_{0, i}^{r+1})$, 
    \begin{align*}
            &\cL\Big( \{\vx_i^{r+1} \}, \{\vlambda_i^{r+1}\}, \{\vx_{0,i}^{r+1}\} \Big) - \cL\Big( \{\vx_i^{r} \}, \{\vlambda_i^r\}, \{\vx_{0, i}^r\} \Big) \\
            &\leq \sum_{i=1}^n -\alpha \|\vx_i^{r+1} - \vx_i^r \|^2 - \sum_{i=1}^n \frac{\beta}{2}\| \vx_{0,i}^{r+1} - \vx_{0,i}^r \|^2 \\
            &+ \sum_{i=1}^n \langle \vlambda_i^r - \beta (\vx_{0,i}^{r+1} - \vx_i^r)-n^{-1}\vt_{0, i}^{r+1},\vx_{0,i}^{r} - \vx_{0,i}^{r+1} \rangle,
    \end{align*}
    
    where $\alpha:=\left( \frac{\beta - L}{2} - \frac{L^2}{\beta} \right)$.
\end{lemma}

\begin{proof}
    We first split the successive difference of the augmented Lagrangian into three simpler-to-analyse terms. 
\begin{subequations}
    \label{aug-lag-decomp}
\begin{align*}
    &\cL\Big( \{\vx_i^{r+1} \}, \{\vlambda_i^{r+1}\}, \{\vx_{0,i}^{r+1}\} \Big) - \cL\Big( \{\vx_i^{r} \}, \{\vlambda_i^r\}, \{\vx_{0, i}^r\} \Big) \\
    &=\underbrace{\cL\Big( \{\vx_i^{r+1} \}, \{\vlambda_i^{r+1}\}, \{\vx_{0, i}^{r+1}\} \Big) - \cL\Big( \{\vx_i^{r+1} \}, \{\vlambda_i^{r}\}, \{\vx_{0, i}^{r+1}\} \Big)}_{\spadesuit} \\
    &+ \underbrace{\cL\Big( \{\vx_i^{r+1} \}, \{\vlambda_i^{r}\}_, \{\vx_{0, i}^{r+1}\} \Big) - \cL\Big( \{\vx_i^{r} \}, \{\vlambda_i^{r}\}, \{\vx_{0, i}^{r+1}\} \Big)}_{\clubsuit} \\
    &+ \underbrace{\cL\Big( \{\vx_i^{r} \}, \{\vlambda_i^{r}\}, \{\vx_{0, i}^{r+1}\} \Big) - \cL\Big( \{\vx_i^{r} \}, \{\vlambda_i^{r}\}, \{\vx_{0, i}^{r}\} \Big)}_{\blacklozenge}
\end{align*}
\end{subequations}
For the first term, using \eqref{aug-lag-case2} we obtain 
    \begin{align}\label{term1_ineq}
        \spadesuit &= \sum_{i=1}^n \langle \vlambda_i^{r+1} - \vlambda_i^{r}, \vx_i^{r+1} - \vx_{0,i}^{r+1} \rangle \nonumber\\
        &\stackrel{(a)}{=} \sum_{i=1}^n \frac{1}{\beta} \|\vlambda_i^{r+1} - \vlambda_i^{r} \|^2 \stackrel{(b)}{\leq} \sum_{i=1}^n \frac{L^2}{\beta} \|\vx_i^{r+1} - \vx_i^r \|^2
    \end{align}
    where in (a), we have used optimality condition of \eqref{x_i_update} and \eqref{dual_update} and in (b), we have used gradient Lipschitzness of $f_i$'s with parameter $L$.
    
        Similarly, using \eqref{aug-lag-case2} for the second term we obtain
    \begin{align}\label{term2_ineq}
        &\clubsuit = \nonumber \\
        &\sum_{i=1}^n \left( f_i(\vx_i^{r+1}) + \langle 
        \vlambda_i^{r} , \vx_i^{r+1} - \vx_{0,i}^{r+1}\rangle + \frac{\beta}{2} \|\vx_i^{r+1} - \vx_{0,i}^{r+1} \|^2\right) \nonumber\\
        &- \sum_{i=1}^n \left( f_i(\vx_i^{r}) + \langle 
        \vlambda_i^{r} , \vx_i^{r} - \vx_{0,i}^{r+1}\rangle + \frac{\beta}{2} \|\vx_i^{r} - \vx_{0,i}^{r+1} \|^2\right) \nonumber\\
        &\stackrel{(a)}{\leq} \sum_{i=1}^n -\frac{\gamma}{2} \|\vx_i^{r+1} - \vx_i^r\|,
    \end{align}
    where in (a) we have used the following inequality
    \begin{align*}
        &l_i^{\beta} (\vx_i^{r+1}, \vlambda_i^{r}, \vx_{0,i}^{r+1}) - l_i^{\beta} (\vx_i^r, \vlambda_i^r, \vx_{0,i}^{r+1}) \\
        & \leq \langle 
        \nabla_{\vx_i} l_i^{\beta}(\vx_i^{r+1}, \vlambda_i^r, \vx_{0,i}^{r+1}), \vx_i^{r+1} - \vx_i^r \rangle - \frac{\gamma}{2} \| \vx_i^{r+1} - \vx_i^r \|^2.
    \end{align*}
    This inequality is a consequence of the strong convexity of the function $l_i^{\beta} (\vx_i, \vlambda_i, \vx_{0,i})$ with respect to its first argument with parameter $\gamma = \beta - L$ for $\beta > L$, and due to the optimality condition on \eqref{x_i_update} which in turn leads to $\nabla_{\vx_i} l_i^{\beta}(\vx_i^{r+1}, \vlambda_i^r, \vx_{0,i}^{r+1}) = 0$. 

    For the last term, we have
    \begin{align}\label{term3_ineq}
        &\blacklozenge = \cL\Big( \{\vx_i^{r} \}, \{\vlambda_i^{r}\}, \{\vx_{0, i}^{r+1}\} \Big) - \cL\Big( \{\vx_i^{r} \}, \{\vlambda_i^{r}\}, \{\vx_{0, i}^{r}\} \Big)\nonumber \\
        &= \sum_{i=1}^n \left( n^{-1}g(\vx_{0,i}^{r+1}) + \langle \vlambda_i^r, \vx_i^r - \vx_{0,i}^{r+1} \rangle +\frac{\beta}{2} \|\vx_i^r - \vx_{0,i}^{r+1} \|^2\right) \nonumber \\
            &- \sum_{i=1}^n \left( n^{-1}g(\vx_{0,i}^{r}) +\langle \vlambda_i^r, \vx_i^r - \vx_{0,i}^{r} \rangle +\frac{\beta}{2} \|\vx_i^r - \vx_{0,i}^{r} \|^2 \right)\nonumber\\
            &\stackrel{(a)}{\leq}\sum_{i=1}^n \left( -\langle n^{-1}\vt_{0,i}^{r+1}-\vlambda_i^r - \beta (\vx_i^r - \vx_{0,i}^{r+1}),\  \vx_{0,i}^r - \vx_{0,i}^{r+1} \rangle \right. \nonumber\\
            &\quad - \left. \frac{\beta}{2} \|\vx_{0,i}^r - \vx_{0,i}^{r+1} \|^2\right).
    \end{align}
    In (a) we have used the fact that the function $\ell_0^{\beta}(\vx) := n^{-1}g(\vx) + \langle \vlambda_i^r, \ \vx_i^r - \vx\rangle + \frac{\beta}{2} \|\vx - \vx_{i}^r \|^2$ is $\beta$-strongly convex due to the convexity of $g$. The result follows from combining \eqref{term1_ineq}, \eqref{term2_ineq}, and \eqref{term3_ineq}.
\end{proof}

\begin{lemma} \label{asymtotically_fixed_vars}
    For the Distributed ADMM algorithm \ref{distributed_admm}. If
    \begin{enumerate}[(i)]
        \item assumptions \ref{smotthnes_f_i}, \ref{proper_f_i} and \ref{g_X_assump} hold;
        \item sequence $( \{\vx_i^r \}_{i=1}^n, \{\vlambda_i^r\}_{i=1}^n, \{\vx_{0, i}^r\}_{i=1}^n)$ is bounded;
        \item the number of consensus steps in iteration $r$ is in order $\cO(\log{r})$ (see Theorem \ref{main_theorem_sync}); and
        \item the penalty parameter $\beta$ satisfies $\frac{\beta - L}{2} - \frac{L^2}{\beta}>0$,
    \end{enumerate}
    then
    \begin{subequations}
    \begin{align}
        &\vx_i^{r+1} - \vx_i^r \to 0, \label{case3-cvg1}\\
        &\vx_{0,i}^{r+1} - \vx_{0,i}^{r} \to 0, \label{case3-cvg2}\\
        &\vlambda_i^{r+1} - \vlambda_i^r \to 0, \label{case3-cvg3}\\
        &\vx_{0,i}^{r} - \vx_i^r \to 0 \label{case3-cvg4}.
    \end{align}
    \end{subequations}
    and the set of its limit points, denoted by $\Lambda^*$, is nonempty and compact and the sequence approaches $\Lambda^*$ as $r \to \infty$.
\end{lemma}

\begin{proof}\label{proof_perturbed_with_g}
    Let $\alpha:=(\frac{\beta - L}{2} - \frac{L^2}{\beta})$. From Lemma \ref{aug-lag-consecutive-change_with_g} we have
    \begin{align}\label{aug_conseq_with_g_zeta}
        &\cL^{r+1} - \cL^r \nonumber \leq \sum_{i=1}^n \left( -\alpha \|\vx_i^{r+1} - \vx_i^r \|^2 -  \frac{\beta}{2}\| \vx_{0 ,i}^{r+1} - \vx_{0,i}^r \|^2\right) \nonumber \\
            &\qquad +  \sum_{i=1}^n\underbrace{\langle \vlambda_i^r - \beta (\vx_{0,i}^{r+1} - \vx_i^r) - n^{-1} \vt_{0,i}^{r+1},\ \vx_{0,i}^{r} - \vx_{0,i}^{r+1} \rangle}_{\zeta_i^r} .
    \end{align}
    Let $\zeta^r := \sum_{i=1}^n \zeta_i^r$. Using \eqref{x_0_i_update_perturbed_with_g}, define $\veta_i^{r+1} := \vx_{0,i}^{r+1} - \vx_0^{r+1}$. Hence, $\veta_i^{r+1} =  \prox_{\frac{1}{n\beta}g}(\widetilde{\vx}^{r+1} + \veps_i^{r+1} ) - \prox_{\frac{1}{n\beta}g}(\widetilde{\vx}^{r+1} )$.
    Note that from non-expansivity of proximal operator (see Lemma \ref{prox_properties_convex}), we infer that $\|\veta_i^{r+1} \| \leq \| \veps_i^{r+1} \|$. Using the equation above, we can write
    \begin{align*}
        \zeta^r &= \sum_{i=1}^n \langle \vlambda_i^r - \beta \vx_{0}^{r+1} {-\beta \veta_i^{r+1}}+ \beta \vx_i^r - n^{-1} \vt_{0,i}^{r+1}, \\  &\qquad \vx_{0}^{r} - \vx_{0}^{r+1} + {\veta_i^{r} - \veta_i^{r+1}\rangle} \\
        & = {\sum_{i=1}^n \langle \vlambda_i^r - \beta \vx_{0}^{r+1} + \beta \vx_i^r - n^{-1} \vt_{0,i}^{r+1},\ {\vx_{0}^{r} - \vx_{0}^{r+1}}\rangle} \\
        &+ {\sum_{i=1}^n \langle \vlambda_i^r - \beta \vx_{0}^{r+1} + \beta \vx_i^r - n^{-1} \vt_{0,i}^{r+1},\  \veta_i^{r} - \veta_i^{r+1}\rangle}\\
        &\quad + {\sum_{i=1}^n \langle  {{-\beta \veta_i^{r+1}}},\  \vx_{0,i}^{r} - \vx_{0,i}^{r+1}\rangle}.
    \end{align*}
    Lets define the three summation terms in the right hand side of the last equality above as $\cK_1$, $\cK_2$, and $\cK_3$, respectively. Next, we bound the values of $\cK_2, \cK_2$, and $\cK_3$. Starting with $\cK_1$, applying Lemma \ref{prox_map_optimality_cond} to the equation \eqref{x_0_i_update_perturbed_with_g}, it follows
    \begin{align*}
        \widetilde{\vx}^{r+1} + \veps_i^r - \vx_{0,i}^{r+1} \in (n\beta)^{-1} \partial g(\vx_{0,i}^{r+1}).
    \end{align*}
    Let $\vt_{0,i}^{r+1} = n\beta \left[ \widetilde{\vx}^{r+1} + \veps_i^{r+1} - \vx_{0,i}^{r+1} \right]$. Thus,
    \vspace*{-.2cm}
    \begin{align*}
        \cK_1 &=  \langle \sum_{i=1}^n \left( \vlambda_i^r - \beta \vx_{0}^{r+1} + \beta \vx_i^r \right. \\
        & \left. - \beta [ \widetilde{\vx}^{r+1} + \veps_i^{r+1} - \vx_{0,i}^{r+1} ] \right),\ \vx_{0}^{r} - \vx_{0}^{r+1}\rangle\\
        & = \langle \beta\sum_{i=1}^n (\vx_{0,i}^{r+1} - \vx_0^{r+1}- \veps_i^{r+1}), \vx_{0}^{r} - \vx_{0}^{r+1} \rangle \\
        &= \beta\langle\sum_{i=1}^n (\veta_i^{r+1} - \veps_i^{r+1}), \vx_{0}^{r} - \vx_{0}^{r+1} \rangle \\
        &\stackrel{(a)}{\leq} 2\beta \| \vx_{0}^{r} - \vx_{0}^{r+1} \| \sum_{i=1}^n \|\veps_i^{r+1} \|
    \end{align*}
    where in (a) we have used the Cauchy-Schwarz inequality and the fact that $\|\veta_i^{r+1} \| \leq \| \veps_i^{r+1} \|$ which implies that $\| \veta_i^{r+1} - \veps_i^{r+1}\| \leq 2 \| \veps_i^{r+1} \|$.
    
    For $\cK_2$ we have\\
     \[    \cK_2 \leq \sum_{i=1}^n\!\! \left( \| \vlambda_i^r - \beta \vx_{0}^{r+1} + \beta \vx_i^r - n^{-1} \vt_{0,i}^{r+1}\| \| \veta_i^{r} - \veta_i^{r+1} \|\right).\]
     
    Using Young's inequality for $\cK_3$ one can write that\\
   $
        \cK_3 \leq \beta \sum_{i=1}^n \|\veta_i^{r+1} - \veta_i^r \|^2 + \frac{\beta}{4} \sum_{i=1}^n \|\vx_{0,i}^{r} - \vx_{0,i}^{r+1} \|^2.
   $   
    Summing up the inequalities above for $\cK_1, \cK_2$, and $\cK_3$ and using \eqref{aug_conseq_with_g_zeta} we have
    \begin{align*}
        &\cL^{r+1} - \cL^r \leq \\ &\sum_{i=1}^n \left( -\alpha \|\vx_i^{r+1} - \vx_i^r \|^2 - \frac{\beta}{4}\| \vx_{0,i}^{r+1} - \vx_{0,i}^r \|^2 \right) + \sum_{i=1}^n \gamma_i^r.
    \end{align*}
    where $\gamma_i^r := 2\beta \| \vx_{0}^{r} - \vx_{0}^{r+1} \|  \|\veps_i^{r+1} \| + \| \vlambda_i^r - \beta \vx_{0}^{r+1} + \beta \vx_i^r - n^{-1} \vt_{0,i}^{r+1}\| \| \veta_i^{r} - \veta_i^{r+1} \|
        + \beta \|\veta_i^{r+1} - \veta_i^r \|^2$.
    By telescoping the inequality above from $r=0$ to $r=T-1$ we obtain
    \begin{align*}
        &\cL^T - \cL^0  \leq \sum_{r=0}^{T-1} \sum_{i=1}^n \left( -\alpha \|\vx_i^{r+1} - \vx_i^r \|^2  \right. \\
        &\quad - \left. \frac{\beta}{4}\| \vx_{0,i}^{r+1} - \vx_{0,i}^r \|^2\right) + \sum_{r=0}^{T-1} \sum_{i=1}^n \gamma_i^r.
    \end{align*}
    
    From the boundedness of the sequences and summability of $\veps_i^{r}$ we can see that $\gamma_i^r$ is summable for all $i\in [n]$, i.e. $\lim_{T\to \infty}\sum_{r=0}^{T-1}\gamma_i^r = \gamma_i < \infty$, by taking limit inferior of the equation above as $T\to \infty$ and using lower boundedness of $\cL$ (Lemma \ref{aug-lag-bdd}) one has
    \begin{align*}
        &\underline{\cL} - \cL^0- \sum_{i=1}^n \gamma_i \leq \liminf_{T \to \infty} \cL^T - \cL^0 - \sum_{i=1}^n \gamma_i \\
        &\leq \liminf_{T \to \infty} \sum_{r=0}^{T-1} \sum_{i=1}^n \left( -\alpha \|\vx_i^{r+1} - \vx_i^r \|^2  - \frac{\beta}{4}\| \vx_{0,i}^{r+1} - \vx_{0,i}^r \|^2\right) \\
        & = -\limsup_{T \to \infty} \sum_{r=0}^{T-1} \sum_{i=1}^n \left( \alpha \|\vx_i^{r+1} - \vx_i^r \|^2 +  \frac{\beta}{4}\| \vx_{0,i}^{r+1} - \vx_{0,i}^r \|^2\right)
    \end{align*}
    which implies
    \begin{align*}
        \limsup_{T \to \infty} \sum_{r=0}^{T-1} \sum_{i=1}^n &\left( \alpha \|\vx_i^{r+1} - \vx_i^r \|^2 + \frac{\beta}{4}\| \vx_{0,i}^{r+1} - \vx_{0,i}^r \|^2\right) \\
        &\leq \cL^0 - \underline{\cL} +\sum_{i=1}^n \gamma_i < \infty.
    \end{align*}
    Therefore, if $\beta$ is chosen such that $\alpha := \frac{\beta - L}{2} - \frac{L^2}{\beta} > 0$, we can conclude that $\lim_{r\to \infty} \| \vx_i^{r+1} - \vx_i^r \| = 0$ and $\lim_{r \to \infty} \| \vx_{0,i}^{r+1} - \vx_{0,i}^r\| = 0$. The optimality condition, \eqref{optim-cond-primal}, the Lipschitz continuity of $f_i$'s, and \eqref{case3-cvg1}, yield \eqref{case3-cvg3} and the dual update \eqref{dual_update} leads to \eqref{case3-cvg4}.
\end{proof}
Now we can prove the main result.
\subsubsection*{Proof of Theorem \ref{main_theorem_sync}} 
    Let $( \{\vx_i^* \}_{i=1}^n, \{\vlambda_i^*\}_{i=1}^n, \{\vx_{0, i}^*\}_{i=1}^n)$ be a limit point of the sequence to which the subsequence $( \{\vx_i^{r_l} \}_{i=1}^n, \{\vlambda_i^{r_l}\}_{i=1}^n, \{\vx_{0, i}^{r_l}\}_{i=1}^n)$ converges as $l\to \infty$.  

    Applying optimality condition on the primal update step \eqref{x_i_update}, gives $\nabla f_i(\vx_i^{r+1}) + \vlambda_i^{r+1} = 0$. Taking limit from both sides over the subsequence, using Lemma \ref{asymtotically_fixed_vars} and Lipschitz continuity of $\nabla f_i$, one can prove \eqref{KKT1}.
    Taking limits from both sides of \eqref{dual_update} over the subsequence and using \eqref{case3-cvg3} in Lemma \ref{asymtotically_fixed_vars} results in $\vx_{0,i}^* = \vx_i^*$ for all $i\in [n]$. Lemma \ref{prox_map_optimality_cond} applied to \eqref{x_0_i_update_perturbed_with_g} yields
    \begin{align*}
       \widetilde{\vx}^{r+1} +\veps_i^{r+1}- \vx_{0,i}^{r+1} \in (n\beta)^{-1} \partial g(\vx_{0,i}^{r+1}), \quad \forall i\in[n].
    \end{align*}
    Thus, there exists $\vs_{0,i}^{r+1} \in \partial g(\vx_{0,i}^{r+1})$ such that $\vs_{0,i}^{r+1} = \sum_{j=1}^n \vlambda_j^r + \beta (\sum_{j=1}^n \vx_{j}^r -n\vx_{0,i}^r) + n\beta \veps_i^{r+1}$. Passing limit over the subsequence results in that $\vs_{0,i}^{r_l} \to \sum_{j=1}^n \vlambda_j^*$ as $l\to \infty$. To prove $\sum_{i=1}^n \vlambda_i^* \in \partial g(\vx_{0,i}^*) $ we need to prove that $g(\vx_{0,i}^{r_l})\to g(\vx_{0,i}^*)$ and then use Remark \ref{subdiff_properties}. From the convexity of $g$, we can write 
    \begin{align*}
        g(\vx_{0,i}^*) \geq g(\vx_{0,i}^{r+1}) + \langle \vs_{0,i}^{r+1} , \vx_{0,i}^* - \vx_{0,i}^{r+1}\rangle.
    \end{align*}
    Taking limsup from both sides of the inequality above over the subsequence implies that \linebreak $\limsup g(\vx_{0,i}^{r+1}) \leq g(\vx_{0,i}^*)$. This along with the lower semi-continuity of $g$ proves the claim that $g(\vx_{0,i}^{r_l})\to g(\vx_{0,i}^*)$. 
    Taking the limit of \eqref{x_0_i_update_perturbed_with_g} and using the fact that the proximal operator is continuous for proper, lsc, and convex functions (see Lemma \ref{prox_properties_convex}) yields $\vx_{0,i}^* = \vx_{0,j}^*$ for all $i,j \in [n]$. This completes the proof.
\hfill $\blacksquare$

\section{Numerical Results} \label{numerical_results}
To evaluate the performance of our proposed algorithm, we conducted numerical experiments focusing on the sparse PCA problem \cite{richtarik2021alternating, chang2016asynchronous}:
% \begin{equation}
%     \begin{array}{ll}
%         \underset{\vx \in \RR^p}{\mbox{minimize}} &\displaystyle{\sum_{i=1}^n -\|P_i \vx \|^2 + \lambda \|\vx\|_1}
%     \end{array}
% \end{equation}

\begin{equation} 
    \begin{array}{ccc}
                \underset{\vx \in \RR^p, \| \vx \|^2 \leq 1 }{\mbox{minimize}} &\displaystyle{\sum_{i=1}^n -\|P_i \vx \|^2 + \lambda \|\vx\|_1}, 
    \end{array}
\end{equation}
where $\lambda$ is the regularisation parameter for $\ell_1$ penalised problem and each agent $i$ locally possesses a data matrix $P_i \in \RR^{m_i \times p}$. Existing algorithms in the literature, such as those in \cite{hong2017distributed, chang2016asynchronous, scutari2019distributed}, are restricted to scenarios with a central node in the communication graph and cannot handle the decentralised problem. Moreover, each $f_i(\vx) :=-\vx^T P_i^T P_i \vx$ is a smooth concave function, enabling a closed-form solution for the sub-problem \eqref{x_i_update} for  sufficiently large $\beta$.

In our numerical experiment, we set $n=20$, $p=500$, $\lambda = 10$, and $m_i = 100$ for all $i\in [n]$. Each element of matrix $P_i$ was independently generated from a Gaussian distribution $\cN(0, 0.1^2)$. To ensure convergence, we selected the penalty parameter $\beta$ such that $\beta > 2\max_{i\in [n]} \lambda_{\text{max}}(P_i^T P_i)$. The communication graph took the form of a ring graph, and the weight matrix $W$ was constructed using the Metropolis-Hastings algorithm (e.g. see \cite{xiao2007distributed}). As a benchmark, we also implemented the centralised algorithm from \cite{hong2017distributed}, adapted to the synchronous setting. To measure the performance of the algorithms towards stationarity and the agreement over the decision variable, we use the proximal gradient and disagreement gap, defined as $G^r:= \|(\bar{\vx}^r - \prox_{ g}(\bar{\vx}^r - \sum_{i=1}^n \nabla f_i(\bar{\vx}^r))\|$ and $D_r:= \max_{i\in [n]} \|\vx_i^r - \bar{\vx}^r \|$, respectively. The results are presented in Fig~\ref{fig:20_agents}. It can be observed that for different values of $\tau$, the number of consensus steps in each iteration, the algorithm's performance remains comparable to that of the decentralised algorithm. However, note that a small number of steps leads to  the algorithm performance deterioration, emphasising the importance of appropriately tuning this parameter to ensure convergence.
\begin{figure}[ht]
    \centering
    \begin{subfigure}[b]{0.4\textwidth}
        \centering
        \includegraphics[width=\textwidth]{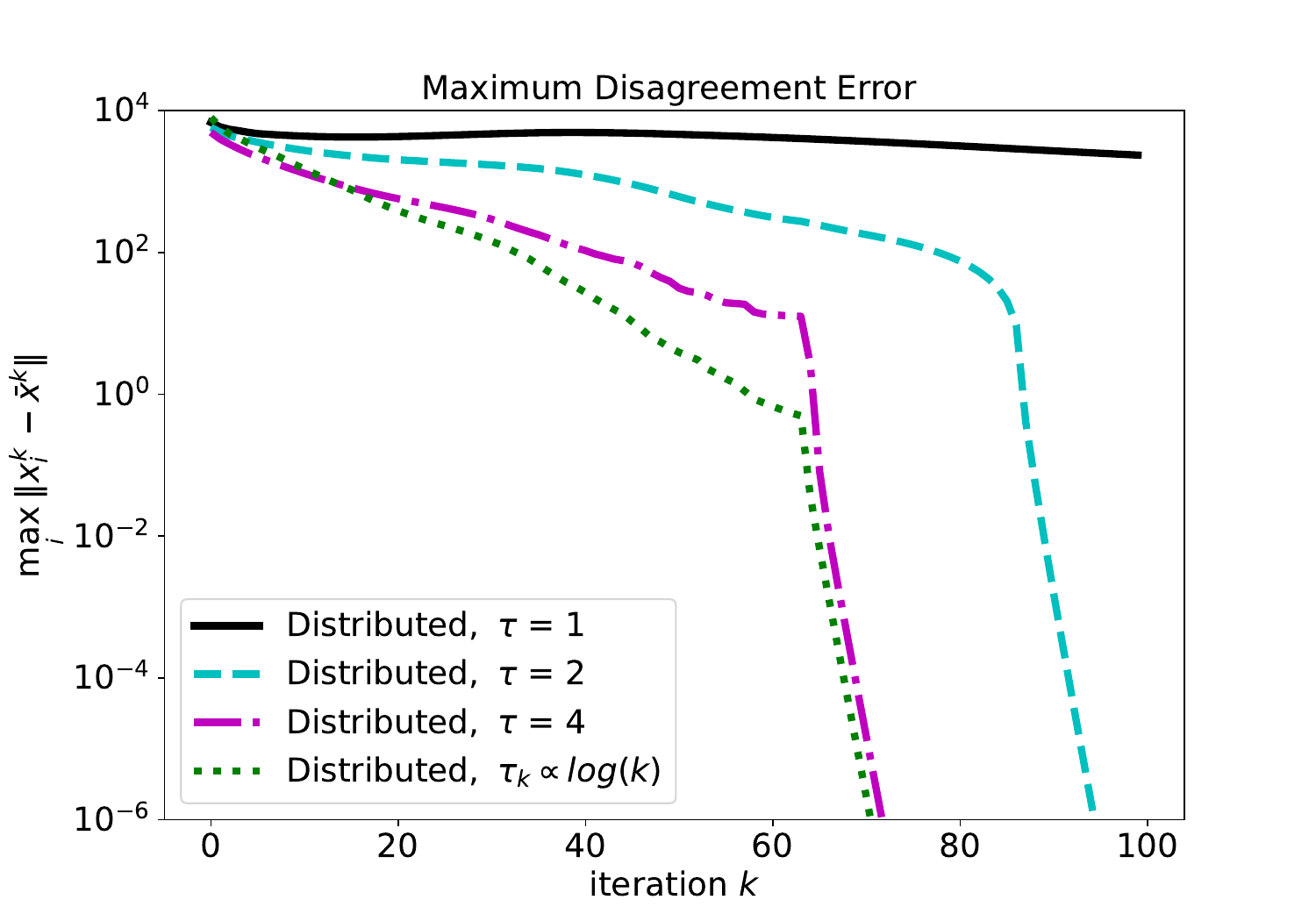}
        \caption{$\log_{10}G^r$ vs. the number of iterations}
        \label{fig:D_r_20}
    \end{subfigure}
    \begin{subfigure}[b]{0.4\textwidth}
        \centering
        \includegraphics[width=\textwidth]{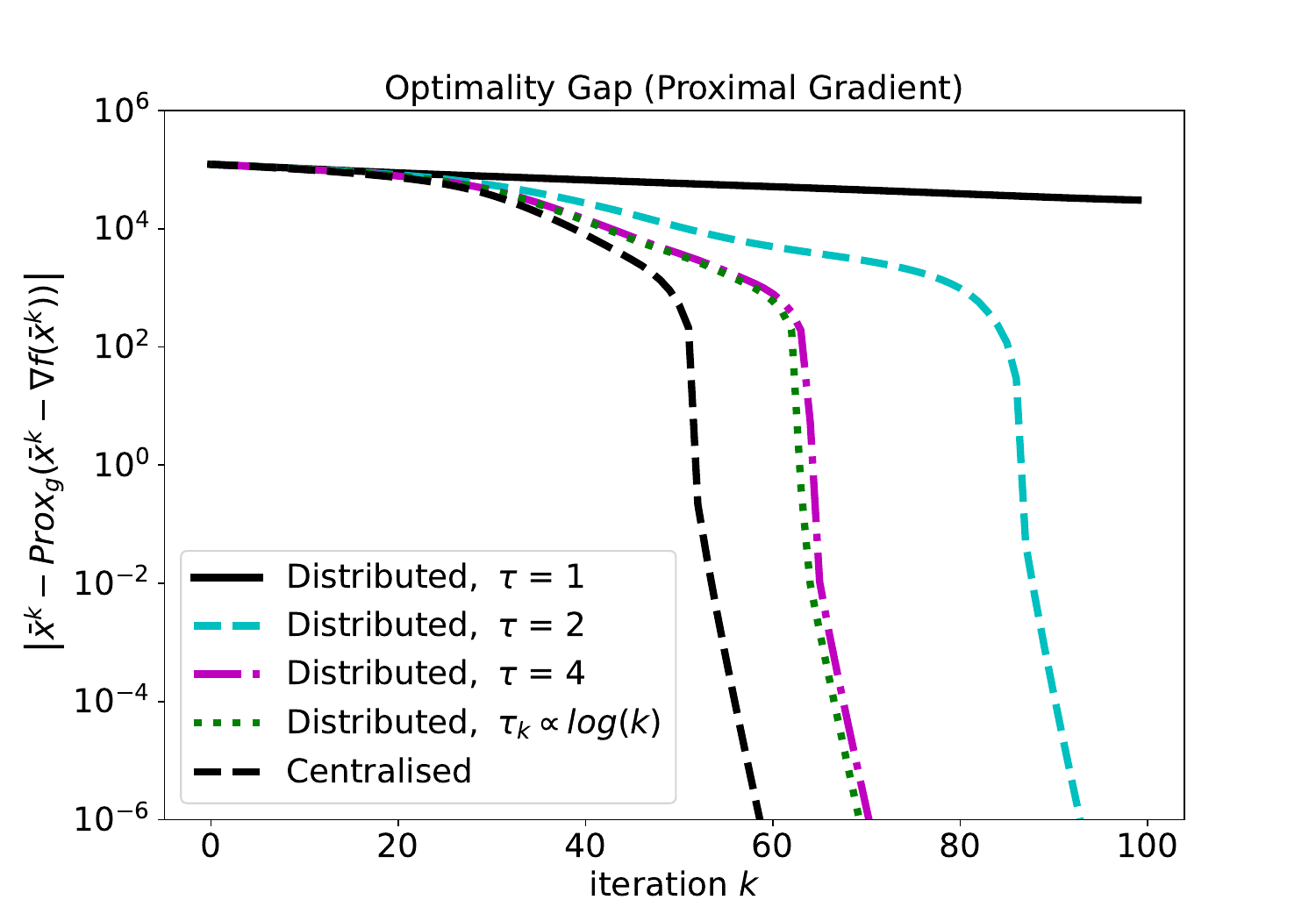}
        \caption{$\log_{10}D^r$ vs. the number of iterations}
        \label{fig:J_r_20}
    \end{subfigure}
    \caption{Distributed PCA problem}%: Centralised \cite{hong2017distributed} and Distributed algorithms}
    \label{fig:20_agents}
\end{figure}
\section{Conclusion}\label{conclusion}
    This paper presents a modified distributed ADMM algorithm designed to tackle nonsmooth nonconvex optimisation problems. Through our analysis, we have demonstrated that by incorporating an adequate number of consensus steps and employing a sufficiently large penalty parameter, the proposed algorithm exhibits convergence to the set of stationary points of the problem. To assess the practical performance of the decentralised ADMM algorithm, we conducted numerical experiments on the sparse PCA problem. The results reveal that even with a small number of inner consensus iterations, the algorithm's performance significantly approaches that of the centralised algorithm in \cite{hong2017distributed}. 
\printbibliography

\end{document}